\documentclass[10pt,a4paper,oneside]{article}
\usepackage{lmodern}
\usepackage[T1]{fontenc}
\usepackage[indentafter]{titlesec}
\titleformat{name=\section}{}{\thetitle.}{0.8em}{\centering\scshape}
\titleformat{name=\subsection}{}{\thetitle.}{0.8em}{\centering\itshape}
\titleformat{name=\subsubsection}[runin]{}{\thetitle.}{0.5em}{\itshape}[.]
\titleformat{name=\paragraph,numberless}[runin]{}{}{0em}{}[.]
\titlespacing{\paragraph}{0em}{0em}{0.5em}
\titleformat{name=\subparagraph,numberless}[runin]{}{}{0em}{}[.]
\titlespacing{\subparagraph}{0em}{0em}{0.5em}
\usepackage{blindtext}

\usepackage{amsthm}

\usepackage{amsfonts}
\usepackage{amssymb}
\usepackage{amssymb,amsmath}
\usepackage{indentfirst}
\usepackage{bussproofs}
\usepackage{graphicx}
\usepackage{pstricks,pst-node,pst-tree, pstricks-add}
\usepackage{color}
\usepackage{fancybox}
\usepackage{bussproofs}
\usepackage{latexsym}
\usepackage{pstricks,pst-node,pst-tree, pstricks-add}
\usepackage{mathtools}
\usepackage{slashed}
\usepackage{picinpar}
\usepackage{float}
\usepackage{mathdots}
\usepackage{tikz-cd}
\usepackage{xypic}
\usepackage[all,2cell,cmtip]{xy} 
\UseAllTwocells 
\usepackage{relsize}
\usepackage{stmaryrd}
\usepackage{longtable}

\definecolor{mlblue}{RGB}{34,139,163}
\definecolor{lblue}{RGB}{63,194,219}

\usepackage{hyperref}
\hypersetup{
    bookmarks=false,         
    unicode=false,          
    pdftoolbar=false,        
    pdfmenubar=false,        
    pdffitwindow=false,     
    pdfstartview={FitH},    
    pdfauthor={Cosimo},     
    colorlinks=true,       
    linkcolor=mlblue,          
    citecolor=lblue,        
	urlcolor=mlblue
}

\newtheorem{thm}{Theorem}[subsection]
\newtheorem{cor}[thm]{Corollary}
\newtheorem{lem}[thm]{Lemma}
\newtheorem{prop}[thm]{Proposition}
\theoremstyle{definition}
\newtheorem{defn}[thm]{Definition}
\theoremstyle{remark}
\newtheorem{oss}{Remark}

\author{Cosimo Perini Brogi}\date{\small{Department of Mathematics $\cdot$ University of Genova} \\ \small\texttt{perinibrogi@dima.unige.it}}
\title{\Large\textsc{Curry-Howard-Lambek correspondence for intuitionistic belief}
\thanks{The present version is a revised one; the original manuscript was submitted for publication to \emph{Studia Logica} in January 2020.}
}

\begin{document}
\maketitle
\tableofcontents
\begin{abstract} This paper introduces a \emph{natural deduction calculus for intuitionistic logic of belief} $\mathsf{IEL}^{-}$ which is easily turned into a \emph{modal $\lambda$-calculus} giving a computational semantics for deductions in $\mathsf{IEL}^{-}$. By using that interpretation, it is also proved that $\mathsf{IEL}^{-}$ has \emph{good proof-theoretic properties}. The correspondence between deductions and typed terms is then extended to a \emph{categorical semantics} for identity of proofs in $\mathsf{IEL}^{-}$ showing the general structure of such a modality for belief in an intuitionistic framework.  
\end{abstract}
\begin{small}
\begin{itemize}
\item[\textbf{Keywords:}] Intuitionistic modal logic, epistemic logic, categorial proof theory, modal type theory, proofs-as-programs.
\end{itemize}
\end{small}
\normalsize

\section*{Introduction}\addcontentsline{toc}{section}{Introduction}

Brouwer-Heyting-Kolmogorov (BHK) interpretation is based on a semantic reading of propositional variables as problems (or tasks), and of logical connectives as operations on \emph{proofs}. In this way, it provides a semantics of mathematical statements in which the computational aspects of proving and refuting are highlighted.\footnote{The reader is referred to \cite{bhk} for an introduction.}

In spite of being named after L.E.J.~Brouwer, this approach is rather away from the deeply philosophical attitude at the origin of intuitionism: In BHK interpretation, reasoning intuitionistically is similar to a safe mode of program execution which always terminates; on the contrary, according to the founders of intuitionism, at the basis of the mathematical activity there is a continuous mental process of construction of objects starting with the flow of time underlying the chain of natural numbers, and intuitionistic reasoning is what structures that process.\footnote{For instance, Dummett's \cite{eloi} advocates a purely philosophical justification of the whole current of intuitionistic mathematics.}

This reading of the mathematical activity is formally captured by Kripke semantics for intuitionistic logic \cite{kripke}: Relational structures based on pre-orders capture the informal idea of a process of growth of knowledge in time which characterises the mental life of the mathematician.

It is worth-noting that the focuses of these semantics are quite different: BHK interpretation stresses the importance of the concept of proof in the semantics for intuitionistic logic; Kripke's approach highlights the epistemic process behind the provability of a statement.

In \cite{artemov}, Artemov and Protopopescu make use of the BHK interpretation to extend -- in a sense -- the epistemic realm of constructivism: In BHK interpretation we have an implicit notion of proof whose epistemic aspects are modelled by Kripke structures; the construction of a proof for a specific proposition that we carried out as a cumulative mental process gives us sufficient reason for (at least) believing that proposition. It is then possible to cover also traditional epistemic states of belief and knowledge within such a framework, once we recognise the correct clauses for corresponding modal operators: Proving $A$ assures that $A$ is true; proving $\Box A$ is weaker, for we may believe in $A$ even though we do not have a direct proof of $A$ itself.

In \cite{artemov} the starting point is thus a BHK interpretation of epistemic statements in which knowledge and belief are considered as (different) results of a process of verification. The general idea is that a proof of a (mathematical) statement is a most strict type of verification, and that verifying a statement is a sufficient condition for \emph{believing} it. At the same time, \emph{knowing} that a statement is true means, according to this intuitionistic reading, that this very statement cannot be false, since we have a verification of it.

Hence, the proposed intuitionistic account of epistemic states validates a principle of ``constructivity of truth'' $$A\rightarrow \Box A$$ and of ``intuitionistic factivity of knowledge'' $$\Box A\rightarrow\neg\neg A.$$

The latter principle appears particularly promising for applying this formal framework in a broader realm than mathematics: In \cite{artemov} various informal justifications to the proposed account are indeed given, and those authors express the aim of giving new foundations to intuitionistic epistemology.

It appears clear, in any case, that both principles are validated by contemporary mathematical practice: When a statement $A$ is proved, then its very proof can be checked by a human being both by `pencil and paper' and by computer proof-assistants, so that the whole mathematical community can believe/know that what has been given is a proof of $A$ indeed. Also, intuitionistic factivity of knowledge states that whenever we know that $A$ is true in a (constructive, but maybe) not specified sense -- e.g.~we do not have direct access to its proof -- then $A$ cannot be false, if we agree that knowledge is at least consistent belief. Again, this is common in general mathematical practice, when a theorem is known to hold even by those who cannot develop a proof of it by themselves.

In short, it seems that the BHK computational reading of the epistemic operator $\Box$ provides a finer analysis of truth and of epistemic states from the constructive point of view.

In spite of this, the mentioned paper \cite{artemov} covers only axiomatic calculi and Kripke semantics for intuitionistic epistemic logics, so that the stimulating question of considering the computational aspects of epistemic states remained informal and at a very beginning stage.

The present paper, on the contrary, is committed to giving a precise, formal analysis of the computational content of intuitionistic belief.

In order to establish some clear facts, a \textbf{natural deduction} system $\mathsf{IEL}^{-}$ for the intuitionistic logic of belief is developed and designed with the intent of translating it into a functional calculus of $\mathsf{IEL}^{-}$-deductions. In a sense, we define a formal counterpart of Artemov and Protopopescu's reading of the epistemic operator for belief by extending the Curry-Howard correspondence between intuitionistic natural deduction $\mathsf{NJ}$ and simple type theory, to a \textbf{modal $\lambda$-calculus} in which the modal connective on propositions behaves according to a single (term-)introduction rule.

Furthermore, we establish \textbf{normalization} for $\mathsf{IEL}^{-}$ along with other \textbf{proof-theoretic properties}, and, in spite of its simple grammar, we show that there is a surprisingly rich \textbf{categorical structure} behind the calculus: Our $\lambda$-system for $\mathsf{IEL}^{-}$-deductions is sound and complete w.r.t. the class of bi-cartesian closed categories equipped with a monoidal pointed endofunctor whose point is monoidal.

Therefore, by adopting the proofs-as-programs paradigm to give a precise meaning to the motto ``belief-as-verification'' we succeed in:
\begin{itemize}
\item[$\circ$] Designing a natural deduction calculus $\mathsf{IEL}^{-}$ for intuitionistic belief which is well-behaved from a proof-theoretic point of view;
\item[$\circ$] Proving that this calculus corresponds to a modal typed system in which every term has a unique normal form, and the epistemic modality acquires a precise functional interpretation;
\item[$\circ$] Developing a categorical semantics for intuitionistic belief which focuses on identity of proofs, and not simply on provability.
\end{itemize}

The paper is then organised as follows: In Section \ref{1}, the axiomatic calculus $\mathbb{IEL}^{-}$ and its relational semantics are recalled. In Section \ref{2}, we introduce the natural deduction system $\mathsf{IEL}^{-}$ and prove -- syntactically -- that it is logically equivalent to $\mathbb{IEL}^{-}$. Then we investigate some of its proof-theoretic properties: We start by proving that detours can be eliminated from deductions by defining a \mbox{$\lambda$-calculus} with a modal operator which captures in a very natural way the behaviour of the epistemic modality on propositions; then we prove a canonicity lemma for this calculus of proof-terms, from which we derive \emph{syntactic} proofs of the weak and standard disjunction property. Finally, in Section \ref{3}, we give a categorical semantics for $\mathsf{IEL}^{-}$-deduction: After recalling the main lines of Curry-Howard-Lambek correspondence, we prove that deductions define -- up to normalization -- specific categorical structures which subsume Heyting algebras with operators and, at the same time, provide a proof-theoretic semantics for intuitionistic belief.

By means of these results we can also see that some claims in \cite{artemov} concerning a type-theoretic reading of the epistemic operator as the truncation of types are not correct. The belief modality there defined is `weaker' than $\mathsf{inh}:\mathtt{Type}\rightarrow\mathtt{Type}$ because of its type-theoretic -- hence syntactic -- behaviour, validated also from a categorical -- hence semantic -- point of view: Types truncation equips bi-cartesian closed categories with an idempotent monad, while we show that the belief operator we are considering is a more general functor.\footnote{In fact even when thought of as a modal connective of our base language, the epistemic $\Box$ is not an idempotent operator.}

\section{Axiomatic calculus for intuitionistic belief}\label{1}
Let's start by recalling the syntax and relational semantics for the logic of intuitionistic belief as introduced in \cite{artemov}.
\subsection{System $\mathbb{IEL}^{-}$}
\begin{defn} $\mathbb{IEL}^{-}$ is the axiomatic calculus given by:
\begin{itemize}
\item Axiom schemes for intuitionistic propositional logic;
\item Axiom scheme $\mathsf{K}: \Box(A\rightarrow B)\rightarrow\Box A\rightarrow\Box B$;
\item Axiom scheme of co-reflection $A\rightarrow\Box A$;
\item Modus Pones \AxiomC{$A\rightarrow B$}\AxiomC{$A$}\RightLabel{$\mathsmaller{\mathit{MP}}$}\BinaryInfC{$B$}\DisplayProof as the only inference rule.
\end{itemize}
We write $\Gamma\vdash_{\mathbb{IEL}^{-}} A$ when $A$ is derivable in $\mathbb{IEL}^{-}$ assuming the set of hypotheses $\Gamma$, and we write $\mathbb{IEL}^{-}\vdash A$ when $\Gamma=\varnothing$. 
\end{defn}

We immediately have
\begin{prop} The following properties hold:
\begin{itemize}
\item[(i)] Necessitation rule \AxiomC{$A$}\UnaryInfC{$\Box A$}\DisplayProof is derivable in $\mathbb{IEL}^{-}$;
\item[(ii)] The deduction theorem holds in $\mathbb{IEL}^{-}$;
\item[(iii)] $\mathbb{IEL}^{-}$ is a normal intuitionistic modal system.
\end{itemize}
\end{prop}
\begin{proof}
See \cite{artemov}.
\end{proof}

As stated before, this system axiomatizes the idea of belief as the result of verification within a framework in which truth corresponds to provability, accordingly to the Brouwer-Heyting-Kolmogorov interpretation of intuitionistic logic.\footnote{See e.g. \cite{troeslavan}.} Note also that, in this perspective, the co-reflection scheme is valid, while its converse does not hold: If $A$ is true, then it has a proof, hence it is verified; but $A$ can be verified without disclosing a specific proof, therefore the standard epistemic scheme $\Box A\rightarrow A$ is not valid under this interpretation.\footnote{See Williamson's system in \cite{williamson} for an intuitionistic epistemic logic in which the standard epistemic principle is valid; however it is worth noting that this specific logic is not based on the BHK-semantics.}

\subsection{Kripke Semantics for $\mathbb{IEL}^{-}$}
Turning to relational semantics, in \cite{artemov} the following class of Kripke models is given.

\begin{defn} A model for $\mathbb{IEL}^{-}$ is a quadruple $\langle W,\leq,v,E\rangle$ where
\begin{itemize}
\item[$\circ$] $\langle W,\leq,v\rangle$ is a standard model for intuitionistic propositional logic;
\item[$\circ$] $E$ is a binary `knowledge' relation on $W$ such that:
\begin{itemize}
\item[$\cdot$] if $x\leq y$, then $xEy$; and
\item[$\cdot$] if $x\leq y$ and $yEz$, then $xEz$; graphically we have $$
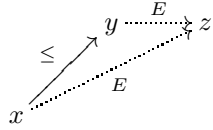
\xymatrix{ & y\ar@{.>}[r]^{E} & z\\
x\ar@{->}[ur]^{\leq}\ar@{.>}[urr]_{E} & & \\
}$$
\end{itemize}
\item[$\circ$] $v$ extends to a forcing relation $\vDash$ such that
\begin{itemize}
\item[$\cdot$] $x\vDash\Box A$ iff $y\vDash A$ for all $y$ such that $xEy$.
\end{itemize}
\end{itemize}
A formula $A$ is true in a model iff it is forced by each world of that model; we write $\mathbb{IEL}^{-}\vDash A$ iff $A$ is true in each model for $\mathbb{IEL}^{-}$.
\end{defn}

Note that this semantics assumes Kripke's original interpretation of intuitionistic reasoning as a growing knowledge -- or discovery process -- for an epistemic agent in which the relation $E$ defines an audit of `cognitively' $\leq$-accessible states in which the agent can commit a verification.

This semantics is adequate to the calculus:
\begin{thm} The following hold:
\begin{description}
\item[(Soundness)] If $\mathbb{IEL}^{-}\vdash A$, then $\mathbb{IEL}^{-}\vDash A$.
\item[(Completeness)] If $\mathbb{IEL}^{-}\vDash A$, then $\mathbb{IEL}^{-}\vdash A$.

\end{description}
\end{thm}
\begin{proof}
Soundness is proved by induction on the derivation of $A$.

Completeness is proved by a standard construction of a canonical model.

See \cite{artemov} for the details. 
\end{proof}
\section{Natural Deduction for intuitionistic belief}\label{2}
We want to develop a \emph{semantics of proofs} for the logic of intuitionistic belief. In order to do that, we now introduce a natural deduction system which is logically equivalent to $\mathbb{IEL}^{-}$, but which is also capable of a computational reading of the epistemic operator and of proofs involving this kind of modality.

Accordingly, the starting point is proving that the calculus $\mathsf{IEL}^{-}$ of natural deduction is sound and complete w.r.t. $\mathbb{IEL}^{-}$; then, we prove that proofs in $\mathsf{IEL}^{-}$ can be named by means of $\lambda$-terms as stated in the proofs-as-programs paradigm for intuitionistic logic also known as Curry-Howard correspondence. By using this formalism, we prove a normalization theorem for $\mathsf{IEL}^{-}$ stating that detours can be eliminated from all deductions. Afterwards, we present some results about the proof-theoretic behaviour of $\mathsf{IEL}^{-}$.

In the next section, we extend such a correspondence to category theory for showing the underlying structure of the operator for intuitionistic belief.

\subsection{System $\mathsf{IEL}^{-}$}
\begin{defn} Let $\mathsf{IEL}^{-}$ be the calculus extending the propositional fragment of $\mathsf{NJ}$ -- the natural deduction calculus for intuitionistic logic as presented in \cite{vandalen} -- by the following rule:
\begin{prooftree} \AxiomC{$\Gamma_{1}$}\noLine\UnaryInfC{$\vdots$}\noLine\UnaryInfC{$\Box A_{1}$}\AxiomC{$\,$}\noLine\UnaryInfC{$\,$}\noLine\UnaryInfC{$\cdots$}\AxiomC{$\Gamma_{n}$}\noLine\UnaryInfC{$\vdots$}\noLine\UnaryInfC{$\Box A_{n}$}\AxiomC{$[A_{1},\cdots , A_{n}],\Delta$}\noLine\UnaryInfC{$\vdots$}\noLine\UnaryInfC{$B$}\singleLine\RightLabel{$\mathsmaller{\Box-intro}$}\QuaternaryInfC{$\Box B$}
\end{prooftree}
where $\Gamma$ and $\Delta$ are \emph{sets of occurrences of formulae}, and all $A_{1},\cdots,A_{n}$ are discharged.

Note that this calculus differs from the system introduced in \cite{depaivaeike} by allowing the set $\Delta$ of additional hypotheses is the subdeduction of $B$.
\end{defn}  

Let's immediately check that we are dealing with the same logic presented in \cite{artemov}
\begin{lem} $\Gamma\vdash_{\mathsf{IEL}^{-}}A$ iff $\Gamma\vdash_{\mathbb{IEL}^{-}}A$.
\end{lem}
\begin{proof}
Assume $\Gamma\vdash_{\mathbb{IEL}^{-}}A$. We proceed by induction on the derivation.
\begin{itemize}
\item Intuitionistic cases are dealt with $\mathsf{NJ}$ propositional rules;
\item $\mathsf{K}$: \AxiomC{$\Box(A\rightarrow B)$}\AxiomC{$\Box A$}\AxiomC{$[A\rightarrow B, A]$}\noLine\UnaryInfC{$\vdots$}\noLine\UnaryInfC{$B$}\singleLine\RightLabel{$\mathsmaller{\Box-intro}$}\TrinaryInfC{$\Box B$}\DisplayProof ;
\item co-reflection: \AxiomC{$[A]^{1}$}\RightLabel{$\mathsmaller{\Box-intro}$}\UnaryInfC{$\Box A$}\RightLabel{$\mathsmaller{\rightarrow-intro:1}$}\UnaryInfC{$A\rightarrow\Box A$}\DisplayProof .
\end{itemize}

Conversely, assume $\Gamma\vdash_{\mathsf{IEL}^{-}}A$. We consider only the $\Box$-intro rule:
By induction hypothesis, we have $\Gamma_{1}\vdash_{\mathbb{IEL}^{-}}\Box A_{1},\cdots,$ \mbox{$\Gamma_{n}\vdash_{\mathbb{IEL}^{-}}\Box A_{n},$} and $A_{1},\cdots,A_{n},\Delta\vdash_{\mathbb{IEL}^{-}}B$. Then we have \mbox{$\Gamma_{1},\cdots,\Gamma_{n}\vdash_{\mathbb{IEL}^{-}}\Box A_{1},\cdots,\Box A_{n},$} and, by the deduction theorem for $\mathbb{IEL}^{-}$ and ordinary logic, \sloppy{\mbox{$\Delta\vdash_{\mathbb{IEL}^{-}}A_{1}\wedge\cdots\wedge A_{n}\rightarrow B$.}} 

By co-reflection \mbox{$\mathbb{IEL}^{-}\vdash(A_{1}\wedge\cdots\wedge A_{n}\rightarrow B)\rightarrow\Box(A_{1}\wedge\cdots\wedge A_{n}\rightarrow B)$,} and by $\mathsf{K}$-scheme $\mathbb{IEL}^{-}\vdash\Box(A_{1}\wedge\cdots\wedge A_{n}\rightarrow B)\rightarrow\Box(A_{1}\wedge\cdots\wedge A_{n})\rightarrow\Box B$. Hence we have $\Delta\vdash_{\mathbb{IEL}^{-}}\Box(A_{1}\wedge\cdots\wedge A_{n})\rightarrow\Box B$, whence, by modal logic, we obtain $\Delta\vdash_{\mathbb{IEL}^{-}}\Box A_{1}\wedge\cdots\wedge\Box A_{n}\rightarrow\Box B$, which gives $\Gamma_{1},\cdots,\Gamma_{n},\Delta\vdash_{\mathbb{IEL}^{-}}\Box B$, as desired.
\end{proof}
\subsection{Normalization}

In order to eliminate potential detours from $\mathsf{IEL}^{-}$-deduction we introduce the following proof rewritings:
\vspace{.28cm}
\begin{longtable}{l l}
\hline
 & \\
($\iota$) & \AxiomC{$\Gamma$}\noLine\UnaryInfC{$\vdots$}\noLine\UnaryInfC{$\Box A$}\AxiomC{$[A]$}\RightLabel{$\mathsmaller{\Box-intro}$}\BinaryInfC{$\Box A$}\DisplayProof$\;\leadsto\qquad\;$\AxiomC{$\Gamma$}\noLine\UnaryInfC{$\vdots$}\noLine\UnaryInfC{$\Box A$}\DisplayProof \\
 & \\
\hline
& \\
($\delta$) & \AxiomC{$\Gamma$}\noLine\UnaryInfC{$\vdots$}\noLine\UnaryInfC{$\vec{\Box A}$}\AxiomC{$[\vec{A}]^{1},\vec{C}$}\noLine\UnaryInfC{$\vdots$}\noLine\UnaryInfC{$B$}\RightLabel{$\mathsmaller{\Box-intro:1}$}\BinaryInfC{$\Box B$}\AxiomC{$\Delta$}\noLine\UnaryInfC{$\vdots$}\noLine\UnaryInfC{$\vec{\Box D}$}\AxiomC{$[B,\vec{D}]^{2},\vec{E}$}\noLine\UnaryInfC{$\vdots$}\noLine\UnaryInfC{$F$}\RightLabel{$\mathsmaller{\Box-intro:2}$}\TrinaryInfC{$\Box F$}\DisplayProof$\leadsto\qquad$  \\ $\qquad$ \\
 & \AxiomC{$\Gamma$}\noLine\UnaryInfC{$\vdots$}\noLine\UnaryInfC{$\vec{\Box A}$}\AxiomC{$\Delta$}\noLine\UnaryInfC{$\vdots$}\noLine\UnaryInfC{$\vec{\Box D}$}\AxiomC{$[\vec{A}]^{1},\vec{C}$}\noLine\UnaryInfC{$\vdots$}\noLine\UnaryInfC{$B$}\AxiomC{$[\vec{D}]^{1},\vec{E}$}\noLine\BinaryInfC{$\vdots$}\noLine\UnaryInfC{$F$}\RightLabel{$\mathsmaller{\Box-intro:1}$}\TrinaryInfC{$\Box F$}\DisplayProof  \\
 & \\
\hline
\end{longtable}
Note that ($\iota$) eliminates a useless application of $\Box$-intro, while ($\delta$) collapses two $\Box$-intros into a single one.
\pagebreak
\begin{oss}[Proofs-as-programs]\label{pasp} Curry-Howard correspondence permits a functional reading of proofs in $\mathsf{NJ}$, once one recognises the following mapping:
\begin{scriptsize}

\vspace{.28cm}
\begin{longtable}{l l l p{3.2cm}}
$f\equiv A$ & $\longmapsto$ & $x_{i}^{A}$, & where $i$ is the parcel of the hypothesis $A$ \\
 & & & \\
$f\equiv$\alwaysNoLine\AxiomC{$f_{1}$}\UnaryInfC{$A$}\AxiomC{$f_{2}$}\UnaryInfC{$B$}\alwaysSingleLine\BinaryInfC{$A\wedge B$}\DisplayProof &  $\longmapsto$ & $\langle t^{A},s^{B}\rangle$, & where $t^{A}$, $s^{B}$ correspond to $f_{1}$ and $f_{2}$ resp. \\
 & & & \\
$f\equiv$\alwaysNoLine\AxiomC{$f'$}\UnaryInfC{$A\wedge B$}\alwaysSingleLine\UnaryInfC{$A$}\DisplayProof &  $\longmapsto$ & $\pi_{1}.t^{A\times B}$, & where $t^{A\times B}$ corresponds to $f'$ \\
 & & & \\
$f\equiv$\alwaysNoLine\AxiomC{$f'$}\UnaryInfC{$A\wedge B$}\alwaysSingleLine\UnaryInfC{$B$}\DisplayProof &  $\longmapsto$ & $\pi_{2}.t^{A\times B}$, & where $t^{A\times B}$ corresponds to $f'$ \\
 & & & \\
$f\equiv$\alwaysNoLine\AxiomC{$f'$}\UnaryInfC{$B$}\alwaysSingleLine\UnaryInfC{$A\rightarrow B$}\DisplayProof &  $\longmapsto$ & $\lambda x_{i}^{A}. t^{B}$, & where $t^{B}$ corresponds to $f'$ and $i$ is the parcel of discharged hypotheses $A$ \\
 & & & \\
$f\equiv$\alwaysNoLine\AxiomC{$f_{1}$}\UnaryInfC{$A\rightarrow B$}\AxiomC{$f_{2}$}\UnaryInfC{$A$}\alwaysSingleLine\BinaryInfC{$B$}\DisplayProof &  $\longmapsto$ & $t^{A\rightarrow B}s^{A}$, & where $t^{A\rightarrow B}$, $s^{A}$ correspond to $f_{1}$ and $f_{2}$ resp. \\
& & & \\
$f\equiv$\alwaysNoLine\AxiomC{$f'$}\UnaryInfC{$A$}\alwaysSingleLine\UnaryInfC{$A\vee B$}\DisplayProof &  $\longmapsto$ & $\;\mathsf{in}_{1}.t^{A}$, & where $t^{A}$ corresponds to $f'$ \\
& & & \\
$f\equiv$\alwaysNoLine\AxiomC{$f'$}\UnaryInfC{$B$}\alwaysSingleLine\UnaryInfC{$A\vee B$}\DisplayProof &  $\longmapsto$ & $\;\mathsf{in}_{2}.s^{B}$, & where $s^{B}$ corresponds to $f'$ \\
& & & \\
$f\equiv$\AxiomC{$f'$}\noLine\UnaryInfC{$A\vee B$}\AxiomC{$[A]$}\noLine\UnaryInfC{$\vdots$}\noLine\UnaryInfC{$C$}\AxiomC{$[B]$}\noLine\UnaryInfC{$\vdots$}\noLine\UnaryInfC{$C$}\TrinaryInfC{$C$}\DisplayProof & $\longmapsto$ & $\mathsf{C}(t,(x^{A}.t_{1}),(y^{B}.t_{2}))$ & where $\mathsf{C}$ bounds all occurrences of $x$ in $t_{1}$ and all occurrences of $y$ in $t_{2}$, and $t, t_{1}, t_{2}$ correspond to $f'$, the subdeduction of $C$ from $A$, and the subdeduction of $C$ from $B$, resp. \\
& & & \\
$f\equiv$\AxiomC{$f'$}\noLine\UnaryInfC{$\bot$}\UnaryInfC{$A$}\DisplayProof & $\longmapsto$ & $\mathsf{E}_{A}t$ & where $t$ corresponds to $f'$ \\
& & & \\
$f\equiv$\AxiomC{$f'$}\noLine\UnaryInfC{$A$}\UnaryInfC{$\top$}\DisplayProof & $\longmapsto$ & $(\mathsf{U}t)$ & where $t$ correspond to $f'$.
\end{longtable}
\end{scriptsize}

By imposing specific rewritings -- which are clearly defined in \cite{girard} -- we obtain the complete engine of \sloppy\mbox{$\lambda$-calculus} associated to the propositional fragment of $\mathsf{NJ}$.

Since $\mathsf{IEL}^{-}$ consists also of a rule for $\Box$, we need to extend the grammar of such a typed $\lambda$-calculus as follows:

$$ T ::= \mathsf{1}\;|\;\mathsf{0}\;|\;p\;|\; A \rightarrow B\; |\; A \times B\; |\; A + B\; |\; \Box A$$
\begin{tabular}{l l}
$t ::=\;$ & $ x\; |\;\mathsf{E}(t)\;|\;\mathsf{U}(t)\;|\; \lambda x:A. t:B\; |\; t_{1}t_{2}\; |\; \langle t_{1},t_{2}\rangle\;|\;\pi_{1}(t)\;|\;\pi_{2}(t)\;|$ \\
 & $\mathsf{in}_{1}(a:A)\;|\;\mathsf{in}_{2}(b:B)\;|\;\mathsf{C}(t, x.t_{1}, y.t_{2})\;|\;$\\ & $\mathsf{box}[\vec{x}:\vec{A}]. \vec{t}:\vec{\Box A} \;\mathsf{in}\; (s:B): \Box B$ .\\ 
\end{tabular}

\vspace{.63cm}

As for $\mathsf{NJ}$, a modal $\lambda$-calculus is obtained by decorating \mbox{$\mathsf{IEL}^{-}$-deductions} with proof names. 

Concerning the modality, we see that in the proof-term \mbox{$\mathsf{box}[\vec{x}:\vec{A}]. \vec{t}:\vec{\Box A} \;\mathsf{in}\; (s:B): \Box B$} for the $\Box$-intro rule, a term $s$ of type $B$ -- with $\vec{x}$ of type $A_{1},\cdots,A_{n}$ among its free variables -- is used along with terms $\vec{t}$ of type $\Box A_{1},\cdots,\Box A_{n}$ in order to obtain a term of type $\Box B$ not depending on the variables $\vec{x}$.

Proof rewritings can be then expressed by imposing appropriate reductions of $\lambda$-terms:\begin{small}
$$\begin{array}{l}
\mathsf{box}[x]. t \;\mathsf{in}\; x >_{\iota} t \\
\mathsf{box}[x_{1},\cdots,x_{i-1},x_{i},x_{i+1},\cdots,x_{n}].(t_{1},\cdots,t_{i-1},(\mathsf{box}[\vec{y}].\vec{s}\;\mathsf{in}\;t_{i}),t_{i+1},\cdots,t_{n})\;\mathsf{in}\; r\,\\ >_{\delta} \mathsf{box}[x_{1},\cdots,x_{i-1},\vec{y},x_{i+1},\cdots,x_{n}].(t_{1},\cdots,t_{i-1},\vec{s},t_{i+1},\cdots,t_{n})\;\mathsf{in}\;r[t_{i}/x_{i}] 
\end{array}$$

\end{small}

Assuming this reading of deductions as programs, normalization now becomes just the execution of a program written in our modal $\lambda$-calculus; it assures then consistency of $\mathsf{IEL}^{-}$, being also useful to prove its analyticity, and hence its decidability. Accordingly, the quest for normalizing natural deduction systems is not limited to the proofs-as-programs paradigm, and its origins are actually at the very core of proof theory: We refer the reader to \cite{structural} and \cite{vonplato} for the technical and historical aspects of the research field, respectively.

\end{oss}

We write $\rhd$ for the transitive closure of the relation obtained by combining $>_{\iota}$ and $>_{\delta}$, together with standard rewritings of $\lambda$-calculus as presented in Remark \ref{pasp}. An algebra of $\lambda$-terms is then obtained by considering the reflexive, symmetric, transitive closure $\overset{\mathsmaller{\Box}}{=}$ of $\rhd$, i.e. by combining the reflexive, symmetric, transitive closure $\overset{\iota}{=}$ and $\overset{\delta}{=}$ of $>_{\iota}$ and $>_{\delta}$, respectively, along with standard rewriting rules for typed $\lambda$-calculus.\footnote{Note that $\overset{\iota}{=}$ is just a special case of $\overset{\delta}{=}$. We decide to adopt this redundant system of rewriting since $\overset{\iota}{=}$ has a straightforward interpretation in category theory: See Section \ref{3.2}.}

We can now prove that every deduction in $\mathsf{IEL}^{-}$ can be uniquely reduced to a proof containing no detours.

\begin{thm}\label{sn} Strong normalization holds for $\mathsf{IEL}^{-}$.
\end{thm}
\begin{proof}
We define a translation\footnote{This function is introduced in \cite{kakutani} to prove detour-elimination for the implicational fragment of basic intuitionistic modal logic $\mathbb{IK}$ by reducing the problem to  normalization of simple type theory. Here we adopt the mapping to consider also product, co-product, empty, and unit types, keeping the original strategy due to \cite{degroote}. A different proof based on Tait's computability method \cite{tait} should also be possible and is under development by the author.} $|-|$ from the $\lambda$-calculus of \mbox{$\mathsf{IEL}^{-}$-deductions} to typed $\lambda$-calculus with products, sums, empty and unit types:
$$\begin{array}{l l l}
|\mathtt{0}| & := & \mathtt{0} \\
|\mathtt{1}| & := & \mathtt{1} \\
|p| & := & p \\
|A\rightarrow B| & := & |A|\rightarrow|B| \\
|A\times B| & := & |A|\times|B| \\
|A+B| & := & |A|+|B| \\
|\Box A| & := & (|A|\rightarrow q)\rightarrow q\\
\end{array}$$
$$\begin{array}{l}
|x| :=  x\\
|c| := c\\
|\mathsf{E}(t)| := \mathsf{E}(|t|)\\
|\mathsf{U}(t)| := \mathsf{U}(|t|)\\
|\lambda x\, t| := \lambda x. |t|\\
|ts| := |t||s|\\
|(t,s)| := (|t|,|s|)\\
|\pi_{i}(t)| := \pi_{i}(|t|)\\
|\mathsf{C}(t, x.t_{1}, y.t_{2})|  := \mathsf{C}(|t|,x.|t_{1}|,y.|t_{2}|)\\
|\mathsf{in}_{i}(t)| :=  \mathsf{in}_{i}(|t|)\\
|\mathsf{box}[x_{1},\cdots,x_{n}]. (t_{1},\cdots,t_{n})\,\mathsf{in}\, s| := \lambda k.|t_{1}|(\lambda x_{1}.\cdots|t_{n}|(\lambda x_{n}.\, k|s|)\cdots)\\
\end{array}$$
where $q$ is specific atom type.
Then it is easy to see that $\overset{\mathsmaller{\Box}}{=}$ is preserved by this mapping: If $t\rhd t'$ in one step, then $|t|\gg|t'|$, where $\gg$ indicates the usual relation made of $\beta\eta$-reductions with permutations. Therefore, since typed \mbox{$\lambda$-calculus} with products, sums, empty and unit types is strongly normalizing,\footnote{See \cite{girard} and \cite{akama}.} so is our modal $\lambda$-calculus, and $\mathsf{IEL}^{-}$ also. In other terms: If $\mathsf{IEL}^{-}$ was not normalizing, then we would have an infinite \mbox{$\rhd$-reduction} starting from, say, $t:A$. By the previous note, this would lead to an infinite \mbox{$\gg$-reduction} starting from $|t|:|A|$, contradicting strong normalization of typed $\lambda$-calculus with $\times,+,\mathtt{1},\mathtt{0}$.
\end{proof}
\begin{lem}\label{cr} The modal $\lambda$-calculus of $\mathsf{IEL}^{-}$-deductions has the Church-Rosser property.
\end{lem}
\begin{proof}
It is straightforward to prove weak Church-Rosser property for our calculus. By Theorem \ref{sn}, the modal $\lambda$-calculus of $\mathsf{IEL}^{-}$-deductions has the Church-Rosser property.\footnote{A proof of Newman's result relating strong normalization and Church-Rosser property is given in \cite[Prop.~3.6.2]{sorensen}.}
\end{proof}
\begin{cor} Every $\mathsf{IEL}^{-}$-deduction reduces uniquely to a deduction without detours.
\end{cor}
\begin{proof}
By Theorem \ref{sn} and Lemma \ref{cr}, any term of the modal \sloppy\mbox{$\lambda$-calculus} of \mbox{$\mathsf{IEL}^{-}$-deductions} has a unique normal form. 
\end{proof}

\subsection{Further Proof-Theoretic Properties}\label{ptheory}

Knowing that $\mathsf{IEL}^{-}$-deductions normalise, it is possible to investigate further on the relevant aspects of our system from the perspective of proof theory.

Following \cite{girard}, we will say that a deduction (or, equivalently, a proof-term) is \textbf{neutral}, iff it is a hypothesis, or its last rule is an elimination rule (E-rule, for short).

We immediately have the following

\begin{lem}\label{key} In any normal and neutral deduction $\Gamma\vdash_{\mathsf{IEL}^{-}} C$, $\Gamma\not =\varnothing$.
\end{lem}
\begin{proof}
Straightforward induction on the height of the deduction.
\end{proof}

By the previous lemma we can easily prove canonicity of proof-terms:

\begin{lem}[Canonicity] If $\mathsf{IEL}^{-}\vdash A$ is normal, then the last rule of the deduction is the appropriate introduction rule. 
\end{lem}
\begin{proof}
By Lemma \ref{key}, since $\Gamma=\varnothing$ here, and the deduction is normal, its last rule cannot be an E-rule.
\end{proof}

It is now straightforward to apply the canonicity lemma in order to obtain the following results:

\begin{cor}[Disjunction Property] If $\mathsf{IEL}^{-}\vdash A\vee B$, then $\mathsf{IEL}^{-}\vdash A$ or $\mathsf{IEL}^{-}\vdash B$.
\end{cor}

\begin{cor} Reflection rule is admissible in $\mathsf{IEL}^{-}$: If $\mathsf{IEL}^{-}\vdash\Box A$, then $\mathsf{IEL}^{-}\vdash A$.
\end{cor}

\begin{cor}[Weak Disjunction Property] If $\mathsf{IEL}^{-}\vdash\Box(A\vee B)$, then $\mathsf{IEL}^{-}\vdash\Box A$ or $\mathsf{IEL}^{-}\vdash\Box B$.
\end{cor}

Another interesting question is whether $\mathsf{IEL}^{-}$ satisfies the subformula property.\footnote{We are extremely grateful to an anonymous reviewer for pointing to this aspect of the structural behaviour of the system.} In the present paper, we are mainly focused on the computational aspects of the intuitionistic modality for belief, and that property would be relevant to these aims only for it implies the decidability of the calculus. We have already developed the main lines of a general proof for the property, and obtained some preliminary results for answering positively the question (modulo minor tweaks); however, proving that the principle holds indeed for $\mathsf{IEL}^{-}$ involve checking many minor details, which make the intermediate proofs quite long. We prefer then to discuss the status of this principle in another work (in preparation) focused on the structural properties of our system. 

\section{Categorical Semantics for intuitionistic belief}\label{3}
If $\lambda$-calculus gives a computational semantics of proofs in $\mathsf{NJ}$ -- and, as we showed in the previous section, in $\mathsf{IEL}^{-}$ also -- category theory furnishes the tools for an `algebraic' semantics which is \emph{proof relevant} -- i.e. contrary to traditional algebraic semantics based on Heyting algebras and to relational semantics based on Kripke models, it focuses on the very notion of proof, distinguishing between different deductions of the same formula.

In this perspective, the correspondence between proofs and programs is extended to consider arrows in categories which have enough structure to capture the behaviour of logical operators. The so-called Curry-Howard-Lambek correspondence can be then summarized by the following table
\begin{center}
\begin{tabular}{c c c}
\hline
\textbf{Logic} & \textbf{Type Theory} & \textbf{Category Theory} \\
\hline
proposition & type & object \\
proof & term & arrow \\
theorem & inhabitant & element-arrow \\
\hline
conjunction & product type & product \\
true & unit type & terminal object \\
implication & function type & exponential \\
disjunction & sum type & coproduct \\
false & empty type & initial object \\
\hline
\end{tabular}
\end{center}
Here we see that cartesian product models conjunction, and exponential models implication. Any category having products and exponentials for any of its objects is called \textbf{cartesian closed (CCCat)}; moreover, if it has also coproducts -- modelling disjunction -- it is called \textbf{bi-cartesian closed (bi-CCCat)}: $\top$ and $\bot$ correspond to empty product \mbox{-- the terminal object $\mathsf{1}$ --} and empty coproduct -- the initial object $\mathsf{0}$ -- respectively. The reader is referred to the classic \cite{lambekscott} for the details of such completeness result.

For our calculus, in order to capture the behaviour of the epistemic modality, some more structure is required: In the following subsections some basic definitions are recalled and then used to provide $\mathsf{IEL}^{-}$ with an adequate categorical semantics.

\subsection{Monoidal Functors, Pointed Functors, and Monoidal Natural Transformations}
\begin{defn}Given a CCCat $\mathcal{C}$, a \textbf{monoidal} endofunctor consists of a functor $\mathfrak{F}:\mathcal{C}\rightarrow\mathcal{C}$ together with
\begin{itemize}
\item a natural transformation $$m_{A,B}: \mathfrak{F}A\times \mathfrak{F}B\rightarrow \mathfrak{F}(A\times B);$$
\item a morphism $$m_{\mathsf{1}}:\mathsf{1}\rightarrow \mathfrak{F}\mathsf{1},$$
\end{itemize}
preserving the monoidal structure of $\mathcal{C}$.\footnote{See \cite{maclane} for the corresponding commuting diagrams and the definition of monoidal category.}

These are called \textbf{structure morphisms} of $\mathfrak{F}$.
\end{defn}

It is quite easy to see that a monoidal endofunctor on the category of logical formulas induces a modal operator satisfying $\mathsf{K}$-scheme, as proved in \cite{depaivaeike}.

\begin{defn}Given any category $\mathcal{C}$, an endofunctor $\mathfrak{F}:\mathcal{C}\rightarrow\mathcal{C}$ is \textbf{pointed} iff there exists a natural transformation $$\begin{matrix}
\pi: Id_{\mathcal{C}}\Rightarrow \mathfrak{F} \\
\pi_{A}: A\rightarrow \mathfrak{F}A
\end{matrix}$$ 
$$\xymatrix{
A \ar@{->}[r]^{\pi_{A}} \ar@{->}[d]_{f} & \mathfrak{F}A\ar@{->}[d]^{\mathfrak{F}f} \\
B \ar@{->}[r]_{\pi_{B}} & \mathfrak{F}B \\
}$$
$\pi$ is called the \textbf{point} of $\mathfrak{F}$.
\end{defn}

In the present setting, a pointed endofunctor on the category of logical formulas `represents' the co-reflection scheme.

Since we want to give a semantics of proofs -- and not simply of derivability -- in $\mathsf{IEL}^{-}$, we need a further notion from category theory.

\begin{defn}Given a monoidal category $\mathcal{C}$, and monoidal endofunctors $\mathfrak{F},\mathfrak{G}:\mathcal{C}\rightarrow\mathcal{C}$, a natural transformation $\kappa:\mathfrak{F}\Rightarrow \mathfrak{G}$ is \textbf{monoidal} when the following commute:
$$\xymatrix{
\mathfrak{F}A\times \mathfrak{F}B \ar@{->}[d]_{\kappa_{A}\times\kappa_{B}}\ar@{->}[r]^{m^{\mathfrak{F}}_{A,B}} & \mathfrak{F}(A\times B)\ar@{->}[d]^{\kappa_{A\times B}} \\
\mathfrak{G}A\times \mathfrak{G}B \ar@{->}[r]_{m^{\mathfrak{G}}_{A,B}} & \mathfrak{G}(A\times B)\\
}$$
and
$$\xymatrix{
\mathsf{1} \ar@{=}[d]\ar@{->}[r]^{m^{\mathfrak{F}}_{\mathsf{1}}} & \mathfrak{F}(\mathsf{1})\ar@{->}[d]^{\kappa_{\mathsf{1}}} \\
\mathsf{1} \ar@{->}[r]_{m^{\mathfrak{G}}_{\mathsf{1}}} & \mathfrak{G}(\mathsf{1})\\
}$$
\end{defn}
\subsection{Categorical Completeness}\label{3.2}

Finally, we introduce the models by which we want to capture $\mathsf{IEL}^{-}$.

\begin{defn} An $\mathsf{IEL}^{-}$-category is given by a bi-CCCat $\mathcal{C}$ together with a monoidal pointed endofunctor $\mathfrak{K}$ whose point $\kappa$ is monoidal.
\end{defn}

Now we can check adequacy of these models.
\begin{thm}[Soundness] Let $\mathcal{C}$ be an $\mathsf{IEL}^{-}$-category. Then there is a canonical interpretation $\llbracket -\rrbracket$ of $\mathsf{IEL}^{-}$ in $\mathcal{C}$ such that

\begin{itemize}
\item[$\circ$] a formula $A$ is mapped to a $\mathcal{C}$-object $\llbracket A\rrbracket$;
\item[$\circ$] a deduction $t$ of $A_{1},\cdots,A_{n}\vdash_{\mathsf{IEL}^{-}}B$ is mapped to an arrow \sloppy\mbox{$\llbracket t\rrbracket:\llbracket A_{1}\rrbracket\times\cdots\times\llbracket A_{n}\rrbracket\rightarrow\llbracket B\rrbracket$;}
\item[$\circ$] for any two deductions $t$ and $s$ which are equal modulo $\overset{\mathsmaller{\Box}}{=}$, we have \mbox{$\llbracket t\rrbracket=\llbracket s\rrbracket$.}
\end{itemize}
\end{thm}
\begin{proof}
By structural induction on $f: \vec{A}\vdash_{\mathsf{IEL}^{-}} B$. The intuitionistic cases are interpreted according to the remarks about bi-CCCats at the beginning of this section. We overload the notation using $\langle\Box,m,\kappa\rangle$ for the monoidal pointed endofunctor of $\mathcal{C}$, its structure morphisms, and its point.

The deduction \begin{center}\begin{scriptsize}

\AxiomC{$f_{1}:\;\Gamma_{1}\vdash\Box A_{1}$}\AxiomC{$\cdots$}\AxiomC{$f_{n}:\;\Gamma_{n}\vdash\Box A_{n}$}\AxiomC{$g:\; [A_{1},\cdots , A_{n}],C_{1},\cdots,C_{m}\vdash B$}
\QuaternaryInfC{$\Box B$}\DisplayProof
\end{scriptsize}
\end{center}
is mapped to $$(\Box\llbracket g\rrbracket)\circ m_{\llbracket A_{1}\rrbracket,\cdots,\llbracket A_{n}\rrbracket,\llbracket C_{1}\rrbracket,\cdots,\llbracket C_{m}\rrbracket}\circ\llbracket f_{1}\rrbracket\times\cdots\times\llbracket f_{n}\rrbracket\times\kappa_{\llbracket C_{1}\rrbracket}\times\cdots\times\kappa_{\llbracket C_{m}\rrbracket},$$ where $m_{X_{1},\cdots,X_{n}}$ is defined inductively as $$m_{X_{1},\cdots,X_{n-1},X_{n}}:=m_{X_{1}\times\cdots\times X_{n-1},X_{n}}\circ(m_{X_{1},\cdots,X_{n-1}})\times id_{\Box X_{n}}.$$

It is straightforward to check that $\overset{\iota}{=}$ holds in the category $\mathcal{C}$ by functoriality of $\Box$.

The relation $\overset{\delta}{=}$ is also valid by naturality of $m$ and $\kappa$: The reader is invited to check that $\kappa$ must be monoidal in order to model correctly the following special case\footnote{Everything reduces to long categorical calculations.}
\begin{center}
\begin{tabular}{l}
\AxiomC{$\Gamma_{1}$}\noLine\UnaryInfC{$\vdots$}\noLine\UnaryInfC{$A_{1}$}\RightLabel{$\mathsmaller{\Box-intro}$}\UnaryInfC{$\Box A_{1}$}\AxiomC{$\,$}\noLine\UnaryInfC{$\,$}\noLine\UnaryInfC{$\cdots$}\AxiomC{$\Gamma_{n}$}\noLine\UnaryInfC{$\vdots$}\noLine\UnaryInfC{$A_{n}$}\RightLabel{$\mathsmaller{\Box-intro}$}\UnaryInfC{$\Box A_{n}$}\AxiomC{$[A_{1},\cdots,A_{n}]^{1},C_{1},\cdots,C_{m}$}\noLine\UnaryInfC{$\vdots$}\noLine\UnaryInfC{$B$}\RightLabel{$\mathsmaller{\Box-intro:1}$}\QuaternaryInfC{$\Box B$}\DisplayProof$\;\leadsto$ \\ $\;$\\ $\leadsto\qquad\qquad$\AxiomC{$\Gamma_{1}$}\noLine\UnaryInfC{$\vdots$}\noLine\UnaryInfC{$A_{1}$}\AxiomC{$\cdots$}\noLine\UnaryInfC{$\,$}\noLine\UnaryInfC{$\,$}\noLine\UnaryInfC{$\,$}\noLine\UnaryInfC{$\,$}\noLine\UnaryInfC{$\cdots$}\AxiomC{$\Gamma_{n}$}\noLine\UnaryInfC{$\vdots$}\noLine\UnaryInfC{$A_{n}$}\AxiomC{$C_{1},\cdots,C_{m}$}\noLine\QuaternaryInfC{$\vdots$}\noLine\UnaryInfC{$B$}\RightLabel{$\mathsmaller{\Box-intro}$}\UnaryInfC{$\Box B$}\DisplayProof
.
\end{tabular}

\end{center} 
\end{proof}

It remains to show that this interpretation is also complete.
\begin{thm}[Completeness] If the interpretation of two \sloppy\mbox{$\mathsf{IEL}^{-}$-deductions} is equal in all \sloppy\mbox{$\mathsf{IEL}^{-}$-categories,} then the two deductions are equal modulo $\overset{\mathsmaller{\Box}}{=}$.
\end{thm}
\begin{proof}
We proceed by constructing a term model for the modal $\lambda$-calculus for $\mathsf{IEL}^{-}$-deductions.
Consider the following category $\mathcal{M}$:
\begin{itemize}
\item its objects are formulae;
\item an arrow $f:A\rightarrow B$ is an $\mathsf{IEL}^{-}$-deduction of $B$ from $A$ modulo $\overset{\mathsmaller{\Box}}{=}$;
\item identities are given by assuming a hypothesis;
\item composition is given by transitivity of deductions.
\end{itemize}

Then $\mathcal{M}$ has a bi-cartesian closed structure given by the properties of conjunction, implication, and disjunction in $\mathsf{NJ}$.

Moreover, the modal operator $\Box$ induces a functor $\mathfrak{K}$ by mapping $A$ to $\Box A$, and

\begin{center}
\AxiomC{$A$}\noLine\UnaryInfC{$\vdots$}\noLine\UnaryInfC{$B$}\DisplayProof$\mapsto$
\AxiomC{$\Box A$}
\AxiomC{$[A]$}\noLine\UnaryInfC{$\vdots$}\noLine\UnaryInfC{$B$}\RightLabel{$\mathsmaller{\Box-intro}$}\BinaryInfC{$\Box B$}\DisplayProof \end{center}

which preserves identities by $\overset{\iota}{=}$ , and preserves composition as a special case of $\overset{\delta}{=}$.

The structure morphism is given by
\begin{center}

\AxiomC{$\Box A\wedge\Box B$}\UnaryInfC{$\Box A$}\AxiomC{$\Box A\wedge\Box B$}\UnaryInfC{$\Box B$}\AxiomC{$[A]$}\AxiomC{$[B]$}\BinaryInfC{$A\wedge B$}\TrinaryInfC{$\Box(A\wedge B)$}\DisplayProof \end{center}
whose properties follow as a special case of $\overset{\delta}{=}$.

The point is given by \AxiomC{$A$}\UnaryInfC{$\Box A$}\DisplayProof and its characteristic property is given as a special case of $\overset{\delta}{=}$. Finally, such a point is monoidal by $\overset{\delta}{=}$ up to $\wedge$-detours.

Then if an equation between interpreted $\mathsf{IEL}^{-}$-deductions holds in all $\mathsf{IEL}^{-}$-categories, then it holds also in $\mathcal{M}$, so that those deductions are equal w.r.t. $\overset{\mathsmaller{\Box}}{=}$.

\end{proof}

\begin{oss}[Belief and truncation] In \cite{bracket}, truncation -- there called ``bracket types'' -- is defined in a first order calculus with types, and showed to behave like a monad. Similarly, in \cite{hottmodalities}, ($n$-)truncation is defined as a monadic idempotent modality within the framework of homotopy type theory.

We have just seen that despite the truncation does eliminate all computational significance to an inhabitant of a type -- turning then a \emph{proof of a proposition into a simple verification of that statement} -- the belief modality defined in \cite{artemov} \emph{does not} correspond to that operator on types.

Actually, after considering the potential applications of $\mathbb{IEL}^{(-)}$ \sloppy{prospected} by Artemov and Protopopescu outside the realm of mathematical statements, that should be not surprising at all: The categorical semantics of $\mathsf{IEL}^{-}$-deductions subsumes the interpretation of truncation as an idempotent monad, since such a functor is just a special case of monoidal pointed endofunctor with monoidal point.

It might be interesting thus to consider the relationship between truncation and the belief modality from a purely syntactic perspective, by comparing the structural properties of a potential simple type theory with bracket types and our modal $\lambda$-calculus for intuitionistic belief.\footnote{As stated before, truncation has been considered only in a first-order context -- i.e. working within dependent types. It should be possible, however, to define truncation for simple types by imposing further reductions to terms of the system mimicking the rules involving (intensional) equality in bracketed types.} 
\end{oss}
\section*{Conclusion}\addcontentsline{toc}{section}{Conclusion}
Our original intent has been to make precise the computational significance of the motto ``belief-as-verification'' which leads in \cite{artemov} to the introduction of epistemic modalities in the framework of BHK interpretation. In particular, despite some claims contained in that paper, we were not sure how to relate the belief operator with type truncation.

In the present paper, we have addressed these questions and have developed a `proof-theoretically tractable' system  for intuitionistic belief that can be easily turned into a modal $\lambda$-calculus, showing that the epistemic operator behaves differently from truncation.

Moreover, by extending some results concerning categorical semantics for the basic intuitionistic modal logic $\mathbb{IK}$ in \cite{depaivaeike} and \cite{kakutani}, we developed a proof-theoretic semantics for intuitionistic belief based on monoidal pointed endofunctors with monoidal points on bi-CCCats. Even from this `categorical' perspective, the modal operator differs from type-theoretic truncation, so that the reading of belief as the result of verification seems to be just a heuristic interpretation of that specific modality.

Having established so, some general questions naturally arise:
\begin{itemize}
\item How could be the original motivation of co-reflection scheme $A\rightarrow\Box A$ -- i.e. the interpretation of $\Box$ as a verification operator on propositions -- correctly captured, from a computational point of view, by intuitionistic logic of belief?
\item Does the possible extension $\mathsf{IEL}$ of $\mathsf{IEL}^{-}$ obtained by adding the elimination rule \AxiomC{$\Gamma\vdash\Box A$}\UnaryInfC{$\Gamma\vdash\neg\neg A$}\DisplayProof recover the intuitionistic reading of epistemic states as results of verification in a formal way -- i.e. as type-theoretic truncation?
\end{itemize}

In our opinion, these problems are strongly related: In fact, it seems plausible that the additional elimination rule provides $\mathsf{IEL}$ with an adjunction between $\Box$ and $\neg\neg$ which has still to be checked and deserves a fine grained analysis and comparison with type truncation.

Moreover, it might be interesting to consider similar modalities in different settings, including first order logic and linear logics.

\subsection*{Acknowledgements}
These results were obtained during my first year of PhD studies at DiMa of University of Genoa.

After discussing a preliminary version of categorical semantics for intuitionistic belief during the Logic and Philosophy of Science Seminar of University of Florence in May 2019, 
I had the great opportunity to present more refined results during the poster session of The Proof Society Summer School in September 2019: My gratitude goes to all logicians of the Computational Foundry of Swansea University who hosted the event, along with the invited lecturers and the participants. In particular, I wish to thank Reuben Rowe for an enlightening discussion on the way to Rhossili about the behaviour of the $\Box$ in my formal system.

The present version is a revised one: The original manuscript was submitted to the journal \emph{Studia Logica} in January 2020.\\
A sincere acknowledgement goes to an anonymous reviewer for suggesting further investigations on the proof theory of the calculus I originally developed: Her/His comments have been very important for structuring the final version of this work.  


 


\end{document}